\documentclass[reqno]{amsart} 

\usepackage[latin1]{inputenc}

\usepackage{amsmath,amssymb,bbm}
\usepackage{enumitem}

\usepackage{epsfig,psfrag,color}

\definecolor{darkgreen}{rgb}{0.00,0.50,0.10}
\definecolor{lightgreen}{rgb}{0.20,0.70,0.30}
\usepackage[colorlinks,bookmarks,linkcolor=darkgreen,citecolor=lightgreen]{hyperref}

\let\subset\subseteq 
\let\eps\varepsilon 
\let\rho\varrho
\let\phi\varphi 
\def\dcup{\dot\cup} 
\def\sh#1{\smash{\hat{#1}}} 
\def\sbar#1{\smash{\bar{#1}}}

\let\l\ell

\def\subsc#1{\textsc{#1}}

\def\itm#1{\rm ({#1})}
\def\itmit#1{\itm{\it #1\,}}

\def\irom{\itmit{\roman{*}}}
\def\iabc{\itmit{\alph{*}}}

\newtheorem{theorem}{Theorem}
\newtheorem{lemma}[theorem] {Lemma}

\newtheorem{proposition}[theorem] {Proposition}

\newtheorem{definition}[theorem] {Definition} 
\theoremstyle{remark}

\newcommand{\Nat}{\ensuremath{\mathbb{N}}}

\newcommand{\Exp}{\mathbb{E}}
\newcommand{\Prob}{\mathbb{P}}

\newcommand{\By}[2]{\overset{\mbox{\tiny{#1}}}{#2}}
\newcommand{\ByRef}[2]{   \By{\eqref{#1}}{#2} }

\newcommand{\leBy}[1]{    \By{#1}{\le} }
\newcommand{\geBy}[1]{    \By{#1}{\ge} }
\newcommand{\eqByRef}[1]{ \ByRef{#1}{=} }
\newcommand{\lByRef}[1]{  \ByRef{#1}{<} }

\newcommand{\geByRef}[1]{ \ByRef{#1}{\ge} }

\newcommand{\bw}{\operatorname{bw}} 

\newcommand{\parti}[2][k]{(#2_i)_{i\in[#1]}}

\title{
Embedding into bipartite graphs
}

\author{Julia B\"ottcher} 
\address{Zentrum Mathematik, Technische Universit\"at M\"unchen,
  Boltzmannstra\ss{}e~3, D-85747 Garching bei M\"unchen, Germany} 
\email{boettche@ma.tum.de}

\author{Peter Heinig}
\address{Zentrum Mathematik, Technische Universit\"at M\"unchen, 
    Boltzmannstra\ss{}e~3, D-85747 Garching bei M\"unchen, Germany} 
\email{heinig@ma.tum.de}

\author{Anusch Taraz } 
\address{Zentrum Mathematik, Technische Universit\"at M\"unchen, 
  Boltzmannstra\ss{}e~3, D-85747 Garching bei M\"unchen, Germany} 
\email{taraz@ma.tum.de}

\thanks{ The first and third author were partially supported by DFG grant
  TA~309/2-1. The second author was supported by TopMath, an Elite Graduate 
  Program of the ENB}

\begin{document}

\begin{abstract}
The conjecture of Bollob\'as and Koml\'os, recently proved by B\"ott\-cher,
Schacht, and Taraz [Math.\ Ann.\ 343(1), 175--205, 2009], implies that for 
any $\gamma>0$, every balanced bipartite graph on $2n$ 
vertices with bounded degree and sublinear bandwidth appears as a subgraph 
of any $2n$-vertex graph $G$ with minimum degree $(1+\gamma)n$, provided that 
$n$ is sufficiently large. We show that this threshold can be cut in half to an 
essentially best-possible minimum degree of $(\frac12+\gamma)n$ when we have the
additional structural information of the host graph $G$ being balanced bipartite. 

This complements results of Zhao~[to appear in SIAM J.\ Discrete Math.], as
well as Hladk\'y and Schacht~[to appear in SIAM J.\ Discrete Math.], who
determined a corresponding minimum degree threshold for $K_{r,s}$-factors,
with $r$ and $s$ fixed. 
Moreover, it implies that the set of Hamilton cycles of~$G$ is a generating
system for its cycle space.

\end{abstract}

\maketitle


\section{Introduction}

The Bollob\'as--Koml\'os conjecture, recently proved in~\cite{BST09}, provides
a sufficient and essentially best possible minimum degree condition for
the containment of $r$-chromatic spanning graphs~$H$ of bounded maximum
degree and small bandwidth.
Here, a graph is said to have bandwidth at most $b$, if there
exists a labelling of the vertices by numbers $1,\dots,n$, such that for every
edge $\{i,j\}$ of the graph we have $|i-j| \le b$.

\begin{theorem}[B\"ottcher, Schacht, Taraz~\cite{BST09}]
\label{thm:bk}
  For all $r,\Delta\in\mathbb{N}$ and $\gamma>0$, there exist constants $\beta>0$
  and $n_0\in\mathbb{N}$ such that for every $n\geq n_0$ the following holds.
  If~$H$ is an $r$-chromatic graph on~$n$ vertices with $\Delta(H) \leq \Delta$
  and bandwidth at most $\beta n$ and if~$G$ is a graph on~$n$ vertices with
  minimum degree $\delta(G) \geq (\frac{r-1}{r}+\gamma)n$, then~$G$ contains a
  copy of~$H$.
\qed
\end{theorem}

This theorem in particular implies that for any
$\gamma>0$, every {bipartite} graph~$H$ on $2n$ vertices with
bounded degree and sublinear bandwidth appears as a subgraph of any $2n$-vertex graph~$G$ with
minimum degree $(1+\gamma)n$, provided that~$n$ is sufficiently large. 
This bound is essentially best possible for an almost trivial reason: there are
graphs~$G$ with minimum degree just slightly below $n$ that are not connected. 
Such~$G$ clearly do not contain a connected~$H$ as a subgraph. These graphs are
simply too different in structure from~$H$.


One may ask, however, whether it is possible to lower the minimum degree
threshold in Theorem~\ref{thm:bk} for graphs~$G$ and~$H$ that are
structurally more similar and, in particular, have the same chromatic number. In
this paper we will pursue this question for the case of balanced bipartite
graphs, i.e., bipartite graphs on $2n$ vertices with~$n$ vertices in each colour class.


Dirac's theorem~\cite{Dir} implies that a $2n$-vertex graph~$G$ with minimum
degree at least~$n$ contains a Hamilton cycle. If $G$ is balanced bipartite,
it follows from a theorem of Moon 
and Moser~\cite{MooMos}
%
%
that this minimum degree threshold can be cut almost in half.

\begin{theorem}
\label{thm:moon}
  Let $G$ be a balanced bipartite graph on $2n$ vertices. If $\delta(G)\geq
  \frac{n}{2}+1$, then $G$ contains a Hamilton cycle.
\end{theorem}

We prove that slightly increasing this minimum degree bound suffices to obtain
all balanced bipartite graphs with bounded maximum degree and sublinear
bandwidth as subgraphs, and thereby establishing the following bipartite
analogue of Theorem~\ref{thm:bk}, halving the minimum degree threshold in that
result.

\begin{theorem}\label{thm:bipbw}
  For all $\gamma$ and $\Delta$ there is a positive constant $\beta$ and an
  integer $n_0$ such that for all $n\ge n_0$ the following holds. Let $G$ and
  $H$ be balanced bipartite graphs on $2n$ vertices such that $G$ has minimum degree
  $\delta(G)\ge(\frac12+\gamma)n$ and $H$ has maximum degree $\Delta$ and
  bandwidth at most $\beta n$. Then $G$ contains a copy of $H$.
\end{theorem}

Results of a similar nature have recently been established by
Zhao~\cite{Zhao_bip}, and by Hladk\'y and Schacht~\cite{HlaSch} who considered
the special case of coverings of~$G$ with disjoint copies of complete
bipartite graphs. Moreover, as a first step towards
Theorem~\ref{thm:bipbw}, in~\cite{heinigbsc} this result was proved for a
special balanced bipartite connected graph (the so-called M\"obius ladder).

We remark that the bandwidth condition in Theorem~\ref{thm:bipbw} cannot be
omitted. Indeed, Abbasi~\cite{Abbasi} proved that the assertion of
Theorem~\ref{thm:bk} gets false if $\beta>4\gamma$. The graph~$H$ he constructs
for this purpose is a balanced bipartite graph and it is not difficult to
see that  Abbasi's host graph contains a bipartite subgraph meeting our
conditions but not containing~$H$. However, the bound on
$\beta$ coming from our proof is very small, having a tower-type dependence
on~$1/\gamma$.

%

The proof of Theorem~\ref{thm:bipbw} is given in Section~\ref{sec:proof}.
It is based on Szemer\'edi's regularity lemma which we introduce in
the following section. In Sections~\ref{sec:G} and~\ref{sec:H} we
provide the proofs of the remaining lemmas that are used in the proof of
Theorem~\ref{thm:bipbw}.


\section{The regularity method}
\label{sec:reg}

In this section we formulate a version of Szemer\'edi's regularity
lemma~\cite{Szemeredi76} that is convenient for our application
(Lemma~\ref{lem:reg}), introduce all necessary definitions, and formulate an
embedding lemma for spanning subgraphs (Lemma~\ref{lem:gel}).

%
%
%

The regularity lemma relies on the concept of a regular pair. To define this, let
$G=(V,E)$ be a graph and $0\le \eps,d\le 1$. For disjoint nonempty vertex sets
$U,W\subset V$ the \emph{density} $d(U,W)$ of the pair $(U,W)$ is the number of
edges that run between $U$ and $W$ divided by $|U||W|$. A pair $(U,W)$ with
density at least $d$ is \emph{$(\eps,d)$-regular} if $|d(U',W')-d(U,W)|\le\eps$
for all $U'\subset U$ and $W'\subset W$ with $|U'|\ge\eps|U|$ and
$|W'|\ge\eps|W|$. The following useful property
of regular pairs follows immediately from the definition.

\begin{proposition}\label{lem:typical}
  Let $G=(A,B)$ be an $(\eps, d)$-regular pair. Let $B'$ be a subset of $B$ with
  $|B'|\ge\eps|B|$. Then there are at most $\eps|A|$ vertices in $A$ with
  less than $(d-\eps)|B'|$ neighbours in $B'$.
\qed
\end{proposition}

The regularity lemma asserts that each graph admits a partition into
relatively few vertex classes of equal size such that most pairs of these
classes form an $\eps$-regular pair. The following definition makes this
precise. A partition $V_0\dcup V_1 \dcup\dotsm\dcup V_k$ of $V$ with
$|V_0|\le\eps|V|$ is \emph{$(\eps,d)$-regular on} a graph $R=([k],E_R)$ if
$ij \in E_R$ implies that $(V_i,V_j)$ is an 
$(\eps,d)$-regular pair in $G$.
If such a partition exists, we also say that $R$ is an 
\emph{$(\eps,d)$-reduced graph} of $G$.
Moreover, $R$ is the \emph{maximal} $(\eps,d)$-reduced graph of the partition
$V_0\dcup V_1 \dcup\dotsm\dcup V_k$ if there is no $ij\not\in E_R$ with
$i,j\in[k]$ such that $(V_i,V_j)$ is $(\eps,d)$-regular.
A partition $V_0\dcup V_1 \dcup\dotsm\dcup V_k$ of $V$ is an
\emph{equipartition} if $|V_i|=|V_j|$ for all $i,j\in[k]$.
The partition classes $V_i$ with $i\in[k]$
are also called \emph{clusters} of $G$ and $V_0$ is the \emph{exceptional set}.
When the exceptional set $V_0$ is empty (or when we want to ignore it as well as
its size) then we may omit it and say that $V_1 \dcup\dots\dcup V_k$ is regular on
$R$.
An $(\eps,d)$-regular pair $(U,W)$ is
\emph{$(\eps,d)$-super-regular} if every vertex $u\in U$ has degree $\deg_W(u)\ge
d|W|$ and every $w\in W$ has $\deg_U(w)\ge d|U|$. For a graph $G=(V,E)$ a
partition $V=V_0\dcup V_1\dcup\dotsm\dcup V_k$ is said to be
\emph{super-regular on a graph $R$} with vertex set $V_R$, $V_R\subset[k]$, if
$(V_i,V_j)$ is super-regular whenever $ij$ is an edge of $R$.

In this paper we consider bipartite graphs and the regular partitions that
appear in the proof of Theorem~\ref{thm:bipbw} refine some bipartition and their
reduced graphs are bipartite. More precisely, for a bipartite graph $G=(A\dcup
B,E)$ we will obtain a partition $(A_0 \dcup B_0)\dcup A_1 \dcup
B_1 \dcup \dots \dcup A_k \dcup B_k$ that is $(\eps,d)$-regular (or
super-regular) on some bipartite graph~$R$ such that $A=A_0 \dcup\dots\dcup
A_k$ and $B=B_0\dcup\dots\dcup B_k$. In particular we have two different
exceptional sets now, one in $A$ called $A_0$ and one in $B$ called $B_0$, each of size~$\eps n$ at most. Such a
partition is an equipartition if $|A_1|=|B_1|=|A_2|=\dots=|A_k|=|B_k|$. In
addition, we consider only regular pairs running between the bipartition
classes, i.e., pairs of the form $(A_i,B_j)$. Consequently, all reduced graphs
(also the maximal reduced graph of a partition) are bipartite.

We now state the version of the regularity lemma that we will use.
This is a corollary of the degree form of the regularity lemma (see,
e.g.,~\cite[Theorem~1.10]{KS96}) and is tailored for embedding applications in
balanced bipartite graphs satisfying some minimum degree condition. We sketch
its proof below.

\begin{lemma}[regular partitions of bipartite graphs]
\label{lem:reg}
For every $\eps'>0$ and for every  $\Delta,k_0\in \Nat$ there
exists $K_0 = K_0(\eps',k_0) \in \Nat$ such that for every $0\leq d' \leq 1$,
for $$
\eps'':= \frac{2\Delta\eps'}{1-\eps'\Delta}
\quad \text{ and } \quad
d'':= d'-2\eps'\Delta\,,
$$
and for every bipartite graph $G = (A\dcup B, E)$ with $|A|=|B|\geq K_0$
and $\delta(G)\ge \nu |G|$ for some $0<\nu<1$ 
there exists a graph $R$ and an integer $k$ with $k_0\le k \le K_0$ with the 
following properties:
\begin{enumerate}[label=\iabc,leftmargin=*]
  \item\label{lem:reg:a} $R$ is an $(\eps',d')$-reduced graph of an
    equipartition of $G$ and $|V(R)|=2k$.
  \item\label{lem:reg:b} $\delta(R) \ge (\nu-d'-\eps'')|R|$.
  \item\label{lem:reg:c} For every subgraph $R^*\subseteq R$ with
    $\Delta(R^*)\le \Delta$ there is an
    equipartition $$
      A \dcup B = A''_0 \dcup B''_0 \dcup A''_1 \dcup B''_1 \dcup \dots \dcup
      A''_k \dcup B''_k 
    $$
    with $A''_i \subseteq A$ and $B''_i \subseteq B$ for all $0\le i \le k$ and
    $(\eps'',d'')$-reduced graph $R$, which in addition is 
    $(\eps'',d'')$-super-regular on $R^*$.
\end{enumerate}
\end{lemma}
\begin{proof}[Proof (sketch)] 
  The proof of this lemma is a standard combination of three standard tools. 
  As a first step we simulate the proof of the degree-form (see \cite{KSS98}, 
  Lemma 2.1, or the survey~\cite{KS96}) of the regularity
  lemma starting with $A\dcup B$ as the
  initial partition (see also~\cite[Chapter~7.4]{Diestel}).
  This yields a partition into clusters $A_0,\dots,B_k$ such that for all
  vertices $v\not\in A_0\cup B_0$ there are at most $(d'+\eps')n$ edges $e\in E$
  with $v\in e$ such that $e$ is not in some $(\eps',d')$-regular pair
  $(A_i,B_j)$. Hence we get~\ref{lem:reg:a}. Let $R$ be the maximal
  (bipartite) $(\eps',d')$-reduced graph of this partition. 
  Then it is easy to see that $R$ inherits the minimum degree
  condition of $G$ (except for a small loss), see
  \cite[Proposition~9]{KueOstTar}. This yields~\ref{lem:reg:b}.
  Finally, for all pairs $(A_i,B_j)$ with $i,j\in[k]$ that correspond to edges
  in $R^*$ we take those vertices in $A_i$ or $B_i$ that have too few
  edges in $(A_i,B_j)$ and move them to $A_0$ or $B_0$, respectively. See
  \cite[Proposition~8]{KueOstTar} for details. This yields~\ref{lem:reg:c}.
\end{proof}

\subsection{Embedding into regular partitions}

For embedding \emph{spanning subgraphs}~$H$ into graphs~$G$ with high minimum
degree the blow-up lemma of Koml\'os, S\'ark\"ozy and Szemer\'edi~\cite{KSS_bl}
has proved to be an extremely valuable tool. 
The blow-up lemma guarantees that bipartite spanning graphs of bounded degree
can be embedded into sufficiently super-regular pairs. In fact, this lemma
is more general and allows the embedding of graphs $H$ into partitions that
are super-regular on some graph $R$ if there is a homomorphism from $H$ to $R$
that does not send too many vertices of $H$ to each cluster of $R$. 

When embedding a spanning graph $H$ into a host graph $G$ a well-established
strategy is to utilise the blow-up lemma on small super-regular
``spots'' in a regular partition of~$G$ for embedding most of the vertices of~$H$,
and to use a greedy embedding method to embed the few other vertices first.
This embedding method is summarised in the next lemma, the general embedding
lemma. Before stating it we need to identify conditions under which it is
possible to proceed in the way just described. This is addressed in the
following definition that specifies when a partition of~$H$ is ``compatible''
with a regular partition of~$G$ with reduced graph~$R$ and a subgraph~$R'$ of~$R$
such that edges of~$R'$ correspond to dense super-regular pairs. In this
definition we require that the partition of~$H$ has smaller partition classes
than the partition of~$G$ (condition~\ref{def:comp:0}), and that edges of~$H$ run
only between partition classes that correspond to a dense regular pair in~$G$
(condition~\ref{def:comp:1}). Further, in each partition class~$W_i$ of~$H$ we
identify two subsets~$S_i$ and~$T_i$ that are both supposed to be small
(condition~\ref{def:comp:2}). The set~$S_i$ contains those vertices that send
edges over pairs that do not belong to the super-regular pairs specified by~$R'$
and~$T_i$ contains neighbours of such vertices.

\begin{definition}[$\eps$-compatible]
\label{def:comp}
  Let $H=(W,E_H)$ and $R=([k],E_R)$ be graphs and let $R'=([k],E_{R'})$ be a
  subgraph of $R$. We say that a vertex partition $W=(W_i)_{i\in[k]}$ of $H$ is
  \emph{$\eps$-compatible} 
  with an integer partition $(n_i)_{i\in[k]}$ of $n$
  and with $R'\subset R$ if the following holds. For
  $i\in[k]$ let $S_i$ be the set of vertices in $W_i$ with neighbours in some
  $W_j$ with $ij\not\in E_{R'}$ and $i\neq j$, set $S:=\bigcup S_i$ and
  $T_i:=N_H(S)\cap (W_i\setminus S)$. Then for all $i,j\in[k]$ we have that
  \begin{enumerate}[label=\irom,leftmargin=*]
    \item\label{def:comp:0}
      $|W_i|\le n_i$,
    \item\label{def:comp:1}
      $xy\in E_H$ for $x\in W_i$ and $y\in W_j$ implies $ij\in E_R$,
    \item\label{def:comp:2}
      $|S_i|\le\eps n_i$ and
      $|T_i|\le\eps\cdot\min\{n_j:\text{$i$ and $j$ are in the same
      component of $R'$}\}$.
  \end{enumerate}
  The partition $W=(W_i)_{i\in[k]}$ of $H$ is $\eps$-compatible
  with a partition $V=(V_i)_{i\in[k]}$ of a graph $G$ and with $R'\subset R$ if 
  $W=(W_i)_{i\in[k]}$ is $\eps$-compatible with $(|V_i|)_{i\in[k]}$ and with
  $R'\subset R$.
\end{definition}

The general embedding lemma asserts that a bounded-degree graph~$H$ can be
embedded into a graph~$G$ if~$H$ and~$G$ have compatible partitions. A proof can
be found in~\cite[Section~3.3.3]{JuliasDiss}.

\begin{lemma}[general embedding lemma]
\label{lem:gel}
  For all $d,\Delta,r>0$ there is a constant $\eps=\eps(d,\Delta,r)>0$ such that
  the following holds.
  Let $G=(V,E)$ be an $n$-vertex graph that has a partition
  $V=(V_i)_{i\in[k]}$ with $(\eps,d)$-reduced graph $R$ on $[k]$ which is
  $(\eps,d)$-super-regular on a graph $R'\subset R$  with connected components 
  having at most $r$ vertices each.
  Further, let $H=(W,E_H)$ be an $n$-vertex graph with maximum degree
  $\Delta(H)\le\Delta$ that has a vertex partition
  $W=(W_i)_{i\in[k]}$ which is $\eps$-compatible with $V=(V_i)_{i\in[k]}$ and
  $R'\subset R$. 
  Then $H\subset G$.
\qed
\end{lemma}

For applying the general embedding lemma to \emph{spanning} graphs~$H$ we need a
partition of the graph~$H$ whose partition classes match the sizes of a
regular partition of~$G$ \emph{precisely}. However, usually we cannot guarantee
that this is the case for a regular partition obtained from
Lemma~\ref{lem:reg}. Hence it will become necessary to modify such a regular
partition slightly by moving some vertices into different clusters. 
The following lemma asserts that the resulting partition is still regular with
somewhat worse parameters.
For a proof see~\cite[Proposition 8]{BoeSchTar_JCTB}.

\begin{proposition}\label{prop:moving-vertices} 
  Let $(A,B)$ be an $(\eps, d)$-regular pair and let $\sh{A}$ and
  $\sh{B}$ be vertex sets with $|\sh{A}\triangle A|\le\alpha|\sh{A}|$
  and $|\sh{B}\triangle B|\le\beta|\sh{B}|$. Then $(\sh{A},\sh{B})$ is
  an $(\sh{\eps},\sh{d})$-regular pair where
  \begin{equation*}
    \sh{\eps} := \eps + 3(\sqrt{\alpha} + \sqrt{\beta}) 
    \qquad \text{and} \qquad \sh{d} := d - 2(\alpha + \beta).
  \end{equation*}
  If, moreover, $(A,B)$ is $(\eps, d)$-super-regular and each vertex $v$ in
  $\sh{A}$ has at least $d |\sh{B}|$ neighbours in $\sh{B}$ and each
  vertex $v$ in $\sh{B}$ has at least $d |\sh{A}|$ neighbours in
  $\sh{A}$, then $(\sh{A},\sh{B})$ is $(\sh{\eps},
  \sh{d})$-super-regular with $\sh{\eps}$ and $\sh{d}$ as above.
\qed
\end{proposition}


\section{The proof of the main theorem}
\label{sec:proof}

In the proof of Theorem~\ref{thm:bipbw} we will use the general embedding lemma
(Lemma~\ref{lem:gel}). For applying this lemma we need compatible
partitions of the graphs $G$ and $H$ which are provided by the next two lemmas.
We start with the lemma for~$G$ which
constructs a regular partition of~$G$ whose reduced graph~$R$ contains a perfect
matching within a Hamilton cycle of~$R$. The lemma guarantees,
moreover, that the regular partition is super-regular on this perfect matching
(see Figure~\ref{fig:G}) and that the cluster sizes in the partition can be
slightly changed.

We remark that, throughout, $A\dcup B$ will denote the vertex set of the
host graph~$G$ while $X\dcup Y$ is the vertex set of the bipartite
graph~$H$ we would like to embed. The sets $A_i$ and $B_i$ with $i\in[k]$ for
some integer $k$ will denote the clusters of a regular partition of $G$ as
well as for the vertices of a corresponding reduced graph.

\begin{lemma}[Lemma for $G$]\label{lem:G}
For every $\gamma > 0$ there exists $d_{\subsc{lg}}>0$
such that for every $\eps>0$ and every $k_0\in \Nat$ there exist
$K_0\in \Nat$   
and $\xi_{\subsc{lg}} > 0$ with the following properties:
For every $n\geq K_0$ and for every balanced bipartite graph $G = (A \dcup B,
E)$ on $2n$ vertices with $\delta(G)\geq \bigl (1/2 + \gamma \bigr) n$
there exists $k_0\leq k\leq K_0$ and a partition $(n_i)_{i\in [k]}$
of $n$ with $n_i\ge n/(2k)$ such that for every partition
$\parti{a}$
of $n$ and $\parti{b}$
of $n$
satisfying $a_i \le n_i + \xi_{\subsc{lg}} n$ and $b_i \le n_i +
\xi_{\subsc{lg}} n$, for all $i\in [k]$, there exist partitions
\[
A= A_1 \dcup \dotsm \dcup A_k
\quad \text{and} \quad
B= B_1 \dcup \dotsm \dcup B_k
\]
such that
\begin{enumerate}[label={\rm(G\arabic{*})}]
\item\label{lem:G:1} $|A_i|=a_i$ and $|B_i|=b_i$ for all $i\in [k]$,
\item\label{lem:G:2} $(A_i, B_i)$ is $(\eps, d_{\subsc{lg}})$-super-regular for every $i\in [k]$.
\item\label{lem:G:3} $(A_i, B_{i+1})$ is $(\eps, d_{\subsc{lg}})$-regular for every $i\in [k]$.
\end{enumerate}
\end{lemma}

\begin{figure}[ht]
    \begin{center}
      \psfrag{A1}{\scalebox{.7}{$A_{\l-1}$}}
      \psfrag{A2}{\scalebox{.7}{$A_\l$}}
      \psfrag{A3}{\scalebox{.7}{$A_{\l+1}$}}
      \psfrag{A4}{\scalebox{.7}{$A_{\l+2}$}}
      \psfrag{B1}{\scalebox{.7}{$B_{\l-1}$}}
      \psfrag{B2}{\scalebox{.7}{$B_\l$}}
      \psfrag{B3}{\scalebox{.7}{$B_{\l+1}$}}
      \psfrag{B4}{\scalebox{.7}{$B_{\l+2}$}}
      \includegraphics{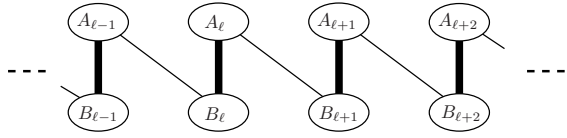}
    \end{center}
    \caption{The regular partition constructed by Lemma~\ref{lem:G} with
      super-regular pairs $(A_i, B_i)$ and regular pairs $(A_i, B_{i+1})$.}
  \label{fig:G}
\end{figure}

The proof of this lemma is presented in Section~\ref{sec:G}. The following
lemma, which we will prove in Section~\ref{sec:H}, constructs the corresponding
partition of~$H$.
It guarantees that the $2k$ partition classes of~$H$ are roughly of the same
sizes as the corresponding partition classes of~$G$ (see~\ref{lem:H:3}), and that all
edges of~$H$ are mapped to edges of a cycle~$C$ on $2k$ vertices and all edges
except those incident to a very small set~$S$ (see~\ref{lem:H:1}) are in fact
mapped to the edges of a perfect matching in~$C$ (see~\ref{lem:H:2}).

\begin{lemma}[Lemma for $H$]\label{lem:H}
For every $k\in \Nat$ 
and every $\xi > 0$ 
there exists $\beta > 0$ 
and $n_0\in \Nat$ 
such that for every $n\geq n_0$ 
and for every balanced bipartite graph $H = (X\dcup Y, F)$ 
on $2n$ vertices 
having $\bw(H)\leq \beta n$ 
and for every integer partition $n=n_1+\dotsm + n_k$ 
with $n_i \leq n/8$ 
there exists a set $S\subseteq V(H)$ 
and a graph homomorphism $f\colon V(H) \rightarrow V(C)$, 
where $C$ is the cycle on the vertices $A_1,B_2,A_2,\dotsc, B_k, A_k, B_1, A_1$, 
such that 
\begin{enumerate}[label={\rm(H\arabic{*})}]
\item\label{lem:H:1} $|S|\leq \xi \cdot 2k \cdot n$,
\item\label{lem:H:2} for every $\{x,y\}\in F$ with 
$x\in X\backslash S$ and $y\in Y\backslash S$ there is $i\in [k]$ such that
$f(x)\in A_i$ and $f(y)\in B_i$,
\item\label{lem:H:3} $|f^{-1}(A_i)| < n_i + \xi n$ and $|f^{-1}(B_i)| < n_i +
\xi n$ for every $i\in [k]$.
\end{enumerate}
\end{lemma}

With Lemmas~\ref{lem:gel} (the general embedding lemma), Lemma~\ref{lem:G} (the
lemma for~$G$) and Lemma~\ref{lem:H} (the lemma for~$H$) at our disposal, we are
ready to give the proof of the main theorem.

\begin{proof}[Proof of Theorem~\ref{thm:bipbw}]
  Given~$\gamma$ and $\Delta$, let $d$ be the constant provided by
  Lemma~\ref{lem:G} for input $\gamma$. Let $\eps$ be the constant
  Lemma~\ref{lem:gel} returns for input $d$, $\Delta$, and $r=2$. We continue the
  application of Lemma~\ref{lem:G} with input $\eps$ and $k_0:=2$ and get
  constants $K_0$ and $\xi_\subsc{lg}$ and set
  $\xi_\subsc{lh}:=\xi_\subsc{lg}\eps/(100\Delta K_0^2)$. Further let $\beta$ be
  the minimum of all the values $\beta_k$ and $n'_0$ be the maximum of all the
  values $n_0^{(k)}$ that Lemma~\ref{lem:H} returns for input $k$ and $\xi$
  where $k$ runs from $k_0$ to $K_0$. Finally, we set $n_0:=\max\{n'_0,K_0\}$.
  
  Let~$G=(A\dcup B,E)$ and~$H=(X\dcup Y,F)$ be balanced bipartite graphs on $2n$
  vertices with $n\ge n_0$, $\delta(G)\ge(\frac12+\gamma)n$,
  $\Delta(H)\le\Delta$, and $\bw(H)\le\beta n$. We apply Lemma~\ref{lem:G} to the
  graph $G$ in order to obtain an integer $k$ and an integer partition
  $(n_i)_{i\in[k]}$ with $n_i\ge \frac12 n/k$ for all $i\in[k]$. Next, we apply
  Lemma~\ref{lem:H} to the graph $H$ and the integer partition $(n_i)_{i\in[k]}$
  and get a vertex set $S\subset X\cup Y$ and a homomorphism $f$ from $H$ to the
  cycle $C$ on vertices $A_1,B_2,A_2,\dots B_k,A_k,B_1,A_1$ such that
  \ref{lem:H:1}--\ref{lem:H:3} are satisfied. With this we can define the integer
  partitions $(a_i)_{i\in[k]}$ and $(b_i)_{i\in[k]}$ required for the
  continuation of Lemma~\ref{lem:G}: set $a_i:=|f^{-1}(A_i)|$ and
  $b_i:=|f^{-1}(B_i)|$ for all $i\in[k]$. By~\ref{lem:H:3} we have $a_i\le
  n_i+\xi_\subsc{lh}n\le n_i+\xi_\subsc{lg}n$ and $b_i\le n_i+\xi_\subsc{lg}n$
  for all $i\in[k]$. It follows that Lemma~\ref{lem:G} now gives us vertex
  partitions $A=(A_i)_{i\in[k]}$ and $B=(B_i)_{i\in[k]}$ for $G$ such that
  \ref{lem:G:1}--\ref{lem:G:3} hold. We complement this with vertex partitions
  $X=(X_i)_{i\in[k]}$ and $Y=(Y_i)_{i\in[k]}$ for $H$ defined by
  $X_i:=f^{-1}(A_i)$ and $Y_i:=f^{-1}(B_i)$ and claim that we can use the general
  embedding lemma (Lemma~\ref{lem:gel}) for these vertex partitions of $G$ and
  $H$.
  
  Indeed, first observe that~\ref{lem:G:2} and~\ref{lem:G:3} imply that the
  partition $V(G)=(A_i)_{i\in[k]}\dcup(B_i)_{i\in[k]}$ is
  $(\eps,d)$-regular on the graph $C$. Further, by~\ref{lem:G:3} this partition
  is $(\eps,d)$-super-regular on the graph $R'$ on the same vertices as $C$ and with
  edges $A_iB_i$ for all $i\in[k]$. Notice that the components of $R'$ have size
  $r=2$. It follows that we can apply Lemma~\ref{lem:gel} if the vertex partition
  $V(H)=(X_i)_{i\in[k]}\dcup(Y_i)_{i\in[k]}$ is $\eps$-compatible with the
  partition $V(G)=(A_i)_{i\in[k]}\dcup(B_i)_{i\in[k]}$ and with $R'\subset C$. To
  check this first note that by~\ref{lem:G:1} we have $|A_i|=a_i=|X_i|$
  and $|B_i|=b_i=|Y_i|$ for all $i\in[k]$ and thus Property~\ref{def:comp:0} of an
  $\eps$-compatible partition is satisfied. Since $f$ is a homomorphism from $H$
  to $C$ we also immediately get Property~\ref{def:comp:1} for
  $(X_i)_{i\in[k]}\dcup(Y_i)_{i\in[k]}$. In addition, since $|A_i|=a_i\le
  n_i+\xi_\subsc{lh}n$ for all $i\in[k]$, we also have $|A_i|\ge
  n_i-k\xi_\subsc{lh}n \ge\frac12n/k-k\xi_\subsc{lh}n\ge \Delta\xi_\subsc{lh}2k
  n/\eps$ by the choice of $\xi_\subsc{lh}$. This together with~\ref{lem:H:1}
  implies that $|S\cap A_i|\le\xi_{\subsc{lh}}2k n\le\eps|A_i|$ and  $|N_H(S)\cap
  A_i|\le\Delta|S|\le\Delta\xi_{\subsc{lh}}2k n\le\eps|A_j|$ for all $i,j\in[k]$.
  Similarly we get $|S\cap B_i|\le\eps|B_i|$ and $|N_H(S)\cap B_i|\le\eps|B_j|$
  for all $i,j\in[k]$. This clearly implies Property~\ref{def:comp:2} of an
  $\eps$-compatible partition.
  
  Accordingly we can apply Lemma~\ref{lem:gel} to the graphs $G$ and $H$ with
  their partitions $V(G)=(A_i)_{i\in[k]}\dcup(B_i)_{i\in[k]}$ and
  $V(H)=(X_i)_{i\in[k]}\dcup(Y_i)_{i\in[k]}$, respectively, which implies that
  $H$ is a subgraph of $G$.
\end{proof}


\section{A regular partition of $G$ with a spanning cycle}
\label{sec:G}

In this section we will prove the Lemma for~$G$. This lemma is a consequence of
the regularity lemma (Lemma~\ref{lem:reg}), Theorem~\ref{thm:moon}, and the
following lemma which states that, under certain circumstances, we can adjust
a (super)-regular partition in order to meet a request for slightly differing cluster sizes.

\begin{lemma}\label{lem:adjust} Let $k\ge 1$ be an integer, 
$0 < \xi \leq 1/(20k^2)$ and let $G = (A\dcup B, E)$ be 
a balanced bipartite graph on $2n$ vertices with
partitions $A=A'_1\dcup\dotsm\dcup A'_k$ and $B=B'_1\dcup\dotsm\dcup B'_k$ such
that $|A'_i|, |B'_i|\geq n/(2k)$ and
$(A'_i,B'_i)$ is $(\eps', d')$-super-regular and $(A'_i, B'_{i+1})$ is 
$(\eps',d')$-regular for all $i\in[k]$. 
Let $\parti{a'}$ and $\parti{b'}$
be integers such that $a'_i,b'_i\le\xi n$ 
for all $i\in[k]$ and $\sum_{i\in [k]}a'_i=\sum_{i\in [k]}b'_i = 0$. 
Then there are partitions
$A=A_1\dcup\dotsm\dcup A_k$ and $B=B_1\dcup\dotsm\dcup B_k$ with
$|A_i|=|A'_i|+a'_i$ and $|B_i|=|B'_i|+b'_i$ and such that $(A_i,B_i)$ is
$(\eps,d)$-super-regular and $(A_i,B_{i+1})$ is $(\eps,d)$-regular for
all $i\in[k]$ where 
$\eps:=\eps' + 100 k\sqrt{\xi}$ and 
$d:=d' - 100 k^2 \sqrt{\xi} - \eps'$.
\end{lemma} 
\begin{proof} 
The lemma will be proved by
performing a simple redistribution algorithm that will iteratively adjust the cluster 
sizes. Throughout the process, we denote by $A_i$ and $B_i$ the 
changing clusters, beginning with $A_i:=A'_i$ and $B_i:=B'_i$.
We call $A_i$ a \emph{sink} when $|A_i|<|A'_i|+a'_i$, and 
a \emph{source} when $|A_i|>|A'_i|+a'_i$, and analogously for $B'_i$.
Each iteration of the algorithm will have the effect that the number 
of vertices in a single source decreases by one, the number of 
vertices in a single sink increases by one, and all other cluster 
cardinalities stay the same. 

We start by describing one iteration of the algorithm. Obviously, 
as long as not every cluster in $A$ has exactly the desired size, 
there is at least one source. We choose an arbitrary source 
$A_i$, and, as will be further explained below, the regularity 
of the pair $(A_i,B_{i+1})$ implies that within $A_i$ there 
is a large set of vertices each of which 
can be added to the neighbouring 
cluster $A_{i+1}$ while preserving  the super-regularity of 
the pair $(A_{i+1},B_{i+1})$. We do this with one arbitrary 
vertex from this set. Thereafter, within $A_{i+1}$ there is 
again a large set of vertices (the newly arrived vertex may 
or may not be one of them) suitable for being moved into $A_{i+2}$ 
while preserving the super-regularity of the pair 
$(A_{i+2},B_{i+2})$, and we again do this with one arbitrary 
vertex from this set. We then continue in this way until for 
the first time we move a vertex into a sink.
(It may happen that it is not the vertex we initially took out of 
$A_i$ that arrives in the sink.) This is the end of the iteration.

We repeat such iterations as long as there are sources, i.e. we 
choose an arbitrary source 
and repeat what we have just described. 
Since each iteration ends with adding a vertex to a sink 
while not changing the cardinality of the clusters visited along 
the way, we do not increase the number of vertices in any source,  
let alone create a new source, and hence after a finite number 
of iterations (which we will estimate below) the algorithm ends with 
no sources remaining and therefore all clusters within $A$ having exactly 
the desired size.

We then repeat what we have just described for the 
clusters within $B$, the only difference being that 
vertices get moved from $B_i$ into $B_{i-1}$, 
not $B_{i+1}$, since only in this direction a regular 
pair can be used ($(A_{i-1},B_i)$ is regular, $(A_{i+1},B_i)$ need not be
regular).

We now analyse the algorithm quantitatively. Clearly, the total 
number of iterations (we call it $t$) is at most the 
sum of all positive $a'_i$ and all positive $b'_i$. Obviously, 
both the sum of all positive $a'_i$ and the sum of all positive 
$b'_i$ is bounded from above by $\frac{1}{2} k \xi n$, hence 
\begin{equation}\label{eq:totalNumberOfChanges} 
t\leq \tfrac{1}{2} k \xi n+\tfrac{1}{2} k \xi n = k\xi n. 
\end{equation} 

We will now use this bound together with 
Proposition~\ref{prop:moving-vertices} to estimate the 
effect of the redistribution on the regularity 
and density parameters. Since in each iteration 
each cluster receives at most one vertex and 
loses at most one vertex, for every $i\in [k]$ and 
after any step of the algorithm, we have 
\[ |A_i\Delta A'_i| \leq 2 t \leq 2 k \xi n\,,\] 
and analogously $|B_i\Delta B'_i|\leq 2 k \xi n$. 
We now invoke Proposition~\ref{prop:moving-vertices} on the pairs 
$(A_i,B_i)$ and $(A_i,B_{i+1})$, once with 
$\sh{A}:=A_i$, $\sh{B}:=B_i$ then with 
$\sh{A}:=A_i$, $\sh{B}:=B_{i+1}$ and 
we claim that we may use $\alpha := \beta := 16 k^2 \xi$. 
Indeed, we have $|A_i|\geq |A'_i| - t \geq n/(2k)-2k\xi n$ and 
because $\xi \leq 1/(20k^2)$ implies $2k\xi n \leq 5 k \xi n - 20k^3\xi^2 n$, 
hence $|A_i\Delta A'_i| \leq 2 k \xi n \leq (5k\xi - 20k^3\xi^2)n 
= 10 k^2\xi (n/(2k) - 2k\xi n) \leq \alpha |A_i'|$, and analogously 
$|B_i\Delta B'_i|\leq \beta |B_i'|$. 
By Proposition~\ref{prop:moving-vertices}, every pair $(A_i,B_i)$ and
$(A_i,B_{i+1})$ is $\bigl( \sh{\eps}, \sh{d} \bigr )$-regular 
with $\sh{\eps}:=\eps' + 24k\sqrt{\xi}$ and $\sh{d}:=d'-64k^2\xi$, 
hence $\sh{\eps}\leq \eps$ and $\sh{d}\geq d$, proving the parameters 
claimed in the lemma, as far as mere regularity goes.

As for the claimed super-regularity of the vertical pairs, 
let $A_i$, $B_i$ and $B_{i+1}$ be clusters at an 
arbitrary step of the algorithm. 
Using Proposition~\ref{lem:typical} and \eqref{eq:totalNumberOfChanges} we 
know that the pairs $(A_i,B_i)$ and $(A_i,B_{i+1})$ being 
$(\sh{\eps},\sh{d})$-regular implies that there are at 
least $(1 - \sh{\eps})|A_i|$ vertices 
in $A_i$ having at least 
$(\sh{d}-\sh{\eps})|B_{i+1}|-t\geq 
(\sh{d}-\sh{\eps})|B_{i+1}| - 2k\xi n$ 
neighbours in $B_{i+1}$, and it remains to 
prove that
$(\sh{d}-\sh{\eps})|B_{i+1}| - 2k\xi n \geq d|B_{i+1}|$ 
which is equivalent to 
$2k\xi n / |B_{i+1}| \leq 100 k^2\sqrt{\xi}-64k^2\xi-24k\xi$. 
Because of $2k\xi n / |B_{i+1}| \leq 2k\xi n/(|B'_{i+1}|-t)
\leq 2k\xi n/(n/2k - 2k\xi n) = 4k^2\xi/(1-4k^2\xi)$ it is 
therefore sufficient that 
$4k^2\xi/(1-4k^2\xi)\leq 100k^2\sqrt{\xi}-64k^2\xi-24k\sqrt{\xi}$ 
and it is easy to check that this is true by the hypothesis 
on~$\xi$.
\end{proof}

Now we will prove Lemma \ref{lem:G}.
To this end we will apply Lemma~\ref{lem:reg} to the input graph~$G$.
By~\ref{lem:reg:a} and~\ref{lem:reg:b} of Lemma~\ref{lem:reg} we obtain a
regular partition with a bipartite reduced graph~$R$ of high minimum degree.
Theorem~\ref{thm:moon} then guarantees the existence of a Hamilton cycle
in~$R$ which will imply property~\ref{lem:G:3}.
This Hamilton cycle serves as~$R^*$ in Lemma~\ref{lem:reg}\ref{lem:reg:c}
which promises a regular partition of~$G$ that is super-regular
on~$R^*$. For finishing the proof we will use a greedy strategy for distributing the
vertices into the exceptional sets over the clusters of this partition (without
destroying the super-regularity required for~\ref{lem:G:2}) and then apply
Lemma~\ref{lem:adjust} to adjust the cluster sizes as needed for~\ref{lem:G:1}.

\begin{proof}[Proof of Lemma~\ref{lem:G}]
Let $\gamma > 0$ given. We assume without loss of generality 
that $\gamma<1/20$ and set $d_{\subsc{lg}}:=\gamma^2/100$. 
Now let $\eps>0$ and $k_0\in\Nat$ be given.
We assume that $\eps\le\gamma^2/1000$, since otherwise we can set
$\eps:=\gamma^2/1000$, prove the lemma, and all statements will still hold for any
larger $\eps$.

Our next task is to choose $\eps'$ and $d'$. 
For this, consider the following functions in $\eps'$ and $d'$:
\begin{equation}\label{lem:G:eps}
\begin{aligned}
\eps'' &:= \frac{\eps'}{1-2\eps'}\,, &\qquad
\sh{\eps} &:= \eps'' + 6\sqrt{\eps''/\gamma(1-\eps'')}\,,
\\ d'' &:= d'-4\eps'\,, &
\sh{d} &:= d'' - 4\eps''/\gamma(1-\eps'')\,.
\end{aligned}
\end{equation}
Observe that
$$
\eps' \ll \eps'' \ll \sh{\eps}\, \qquad\text{and}\qquad
\sh{d} \ll d'' \ll d'\,,
$$
by which we mean, for example, that $\eps'\le \eps''$ but that we can make 
$\eps''$ arbitrarily small by choosing $\eps'$ sufficiently small. 
Keeping in mind that $\gamma <1/20$, it is easy to check that when setting 
$\eps':=\eps^3\gamma^3$ and $d':=\eps+\gamma^2$, the following inequalities are 
all satisfied:
\begin{gather}
\label{eq:lg1}
  \sh{\eps} \le\tfrac1{10}\eps\,, \qquad
  \sh{d} - \eps \ge 2d_{\subsc{lg}}\,, \qquad
  \gamma-d'-\eps'' > 0\, \\
\label{eq:lg2}
  (\tfrac12+\gamma-\eps'')(1-d'')^{-1} \ge \tfrac12  + \tfrac23 \gamma\,,
  \qquad {d''}(1-{d''})^{-1} \le \tfrac16 \gamma\,.
\end{gather}
Next, using \eqref{eq:lg1}, we can choose an integer $k'_0$ with $k_0\le k'_0$
such that for all 
integers $k$ with $k'_0\le k$ we have
\begin{equation}
\label{eq:lg6}
(\gamma-d'-\eps'')k \ge 1\,.
\end{equation}

Apply Lemma~\ref{lem:reg} with $\eps'$, $\Delta:=2$, and with $k_0$
replaced by $k'_0$, to obtain $K_0$. Choose $\xi_{\subsc{lg}}>0$ such that
\begin{equation}
\label{eq:lg7}
100 K_0 \sqrt{\xi_{\subsc{lg}}} \le \tfrac1{10}{\eps}, \quad
100 (K_0)^2 \sqrt{\xi_{\subsc{lg}}} \le d_{\subsc{lg}}.
\end{equation}
Now let $G$ be given. Feed $d'$ and $G$ into Lemma~\ref{lem:reg} and obtain 
$k\in\Nat$ with $k_0 \le k'_0\le k \le K_0$ 
together with an equipartition of $G$ into $2k+2$ classes and an 
$(\eps',d')$-reduced graph $R$ on $2k$ vertices by~\ref{lem:reg:a} of
Lemma~\ref{lem:reg}. By assumption $\delta(G)\ge (\frac12 +\gamma)n$, so setting
$\nu:= 1/2+\gamma$ and making use of part~\ref{lem:reg:b} of Lemma~\ref{lem:reg}, we get
$$
\delta(R) \ge (\tfrac12 + \gamma - d'-\eps'')|V(R)| 
= \tfrac12 |V(R)| + (\gamma-d'-\eps'')k 
\geByRef{eq:lg6} \tfrac12 |V(R)| + 1.
$$ 
We infer from Theorem~\ref{thm:moon} that~$R$ contains a Hamilton cycle~$R^*$. 
Now apply part~\ref{lem:reg:c} of Lemma~\ref{lem:reg} and obtain an 
equipartition of $G$ which is
$(\eps'',d'')$-regular on $R$, $(\eps'',d'')$-super-regular on $R^*$,
and has classes
$$
A= A''_0 \dcup \dots \dcup A''_k 
\quad \text{and} \quad
B= B''_0 \dcup \dots \dcup B''_k. 
$$
Obviously, $R$ and thus $R^*$ are bipartite and so,
without loss of generality
(renumbering the clusters if necessary), we can assume that the Hamilton
cycle $R^*$ consists of the vertices representing the classes
$$
A''_1,B''_2,A''_2,B''_3,\dots,B''_k,A''_k,B''_1,A''_1
$$
with edges in this order. Therefore, we know that
the pairs $(A''_i,B''_i)$ and $(A''_i,B''_{i+1})$
are $(\eps'',d'')$-super-regular for all $i\in [k]$. 
Let $L:=|A''_i|=|B''_i|$
and observe that 
$$
(1-\eps'')\frac{n}{k} \le L \le \frac{n}{k}\,.
$$

Our next aim is to get rid of the classes $A''_0$ and $B''_0$
by moving their vertices to other classes. 
We will do this, roughly speaking, as follows.
When moving a vertex $x\in A''_0$ to some class $A''_i$, say,
we will move an arbitrary vertex $y\in B''_0$ to the corresponding class $B''_i$
at the same time. We will also make sure that $x$ has at least $d'' |B''_i|$
neighbours in $B''_i$ and $y$ has at least $d'' |A''_i|$ neighbours in
$A''_i$.
Here are the details for this procedure.
For an arbitrary pair $(x,y)\in A''_0\times B''_0$ we define
\begin{equation*} 
I_{}(x,y) := \Big\{ i\in [ k ]\colon\quad |N_G (x)\cap B''_i| \geq d'' \; 
|B''_i|\quad\text{and}\quad|N_G (y)\cap A''_i| \geq d'' \; |A''_i|\Big\}\,. 
\end{equation*} 
We claim that for every $(a,b)\in A''_0\times B''_0$ we have
$
|I_{}(x,y)| \ge \gamma k
$.
To prove this claim, first recall  that $L=|A''_i|=|B''_i|$ for all $i\in [k]$.
Define 
\begin{align*}
I_{}(x) &:= \big\{ i\in [ k ]\colon |N_G (x)\cap B''_i| \geq d''  |B''_i|
\big\}\,, \\
I_{}(y) &:= \big\{ i\in [ k ]\colon |N_G (y)\cap A''_i| \geq d''  |A''_i|
\big\}\,.
\end{align*}
As $|A''_0|=|B''_0| \le\eps'' n$ we have
\begin{align*}
(\tfrac12 + \gamma)n 
&\,\le\, \deg_G(x) 
 \,\le\, |I_{}(x)| L + (k-|I_{}(x)|) \,d'' L + \eps'' n \\
&\,=\, |I_{}(x)| (1-d'') L + k d'' L + \eps'' n\,. 
\end{align*}
and hence
\begin{align*} 
|I_{}(x)| 
&\ge \frac{(\frac12+\gamma)n -kd'' L - \eps'' n}{(1-d'')L} 
= \frac{(\frac12+\gamma-\eps'')}{1-d''} \frac{n}{L} 
  -\frac{d''}{1-d''} \,k \\ 
&\geByRef{eq:lg2} (\tfrac12 + \tfrac23\gamma)k -
  \tfrac16\gamma k = (\tfrac12 + \tfrac12\gamma)k \,.
\end{align*}
Similarly, $|I_{}(y)| \ge (\frac12 + \frac12\gamma)k$. 
Since $I_{}(x)$ and $I_{}(y)$ are both subsets of $[k]$, this implies that
$
|I(x,y)|= |I_{}(x) \cap I_{}(y)| \ge \gamma k
$,
which proves the claim. 

We group the vertices in $A''_0 \cup B''_0$ into (at most $\eps'' n$) pairs
$(x,y)\in A''_0\times B''_0$ and 
choose an index $i\in I(x,y)$ which has the property that $(A''_i,B''_i)$ has so
far received a minimal number of additional vertices.
Then we move $x$ into $A''_i$ and $y$ into $B''_i$.
Hence, at the end, every cluster $A''_i$, or $B''_i$ gains at most 
$\eps'' n / (\gamma k)$ additional vertices.
Denote the final partition obtained in this way by 
$$
A \dcup B = \sh{A}_1 \dcup \sh{B}_1 \dcup \dots \dcup \sh{A}_k \dcup \sh{B}_k\,.
$$
Set $\alpha:=\beta:=\eps''/\gamma(1-\eps'')$ and observe that 
$$
\frac{\eps'' n}{\gamma k} = \alpha (1-\eps'')\frac{n}{k} \le \alpha L\,.
$$
So Proposition~\ref{prop:moving-vertices} tells us that for all $i\in [k]$ the
pairs $(\sh{A}_i,\sh{B}_i)$ are still 
$(\sh{\eps},\sh{d})$-super-regular and the pairs 
$(\sh{A}_i,\sh{B}_{i+1})$ are still $(\sh{\eps},\sh{d})$-regular,
because
\begin{align*}
  \sh{\eps}
    &\eqByRef{lem:G:eps}\eps'' + 6\sqrt{\eps''/\gamma(1-\eps'')}
    =\eps''+3(\sqrt{\alpha}+\sqrt{\beta}) 
  \qquad\text{and} \\
  \sh{d}
    &\eqByRef{lem:G:eps} d'' - 4\eps''/\gamma(1-\eps'')
    = d'' -4\alpha 
    = d'' -2(\alpha+\beta) \,.
\end{align*}
Now we return to the statement of Lemma~\ref{lem:G}. 
We set $n_i:=|\sh{A_i}|=|\sh{B_i}|$ for all $i\in [k]$.
Let $\parti{a}$ and $\parti{b}$ be given and set $a''_i:= a_i- n_i$ and
$b''_i:= b_i- n_i$. Then
$$
a''_i \le \xi_\subsc{lg}n, \quad 
b''_i \le \xi_\subsc{lg}n, \quad
\sum_{i\in[k]} a''_i = \sum_{i\in[k]} a_i - \sum_{i\in[k]} n_i = n-n =0 =
\sum_{i\in[k]} b''_i\,. $$
Therefore we can apply Lemma~\ref{lem:adjust} with parameter~$\xi_\subsc{lg}$
to the graph~$G$ with partitions
$\sh{A}_1 \dcup \dots \dcup \sh{A}_k$ and $\sh{B}_1 \dcup \dots \dcup \sh{B}_k$.
Since
\begin{align*}
\sh{\eps} + 100 k \sqrt{\xi_\subsc{lg}} 
  \leBy{\eqref{eq:lg1},\eqref{eq:lg7}}\tfrac{1}{10}\eps+\tfrac{1}{10}\eps
  &\le \eps \qquad\text{and} \\
\sh{d} - 100 k^2 \sqrt{\xi_\subsc{lg}}-\eps 
  \geBy{\eqref{eq:lg1},\eqref{eq:lg7}} 2d_\subsc{lg} - d_\subsc{lg} 
  &= d_\subsc{lg}\,,
\end{align*}
we obtain sets $A_i$ and $B_i$ for each $i\in [k]$ such that
$|A_i| = |\sh{A}_i| + a''_i = n_i + a''_i = a_i$ 
and $|B_i|=b_i$,
and with the property that
$(A_i,B_i)$ is $(\eps,d)$-super-regular and 
$(A_i,B_{i+1})$ is $(\eps,d)$-regular.
This completes the proof of Lemma~\ref{lem:G}.
\end{proof}


\section{Distributing $H$ among the edges of a cycle}
\label{sec:H}

In this section we will provide the proof of the Lemma for~$H$
(Lemma~\ref{lem:H}). The idea is to cut~$H$ into small pieces along its
bandwidth ordering, that is, an ordering of the vertices $H$ that respects
the bandwidth bound. These pieces are then distributed to the edges $A_iB_i$ of
the cycle~$C$ in such a way that the following holds. Let $X_i$ be all the
vertices from $X$, and $Y_i$ all the vertices from $Y$ that were assigned to the
edge $A_iB_i$. Then we require that $X_i$ and $Y_i$ are roughly of size $n_i$.
Observe that this goal would be easy to achieve if~$H$ were \emph{locally
balanced}, i.e., if each of the small pieces had colour classes of equal size.
While this need not be the case, we know, however, that~$H$ itself is a
\emph{balanced} bipartite graph. Therefore we use a probabilistic argument to
show that the pieces of $H$ can be grouped in such a way that the resulting
packages form balanced bipartite subgraphs of $H$. The details of this argument are
given in Section~\ref{subsec:balance}.

After this distribution of the pieces to the edges $A_iB_i$ we will construct
the desired homomorphism~$f$ in the following way. We will map most vertices of
$X_i$ to $A_i$ and most vertices of $Y_i$ to $B_i$. 

\subsection{Balancing $H$ locally}
\label{subsec:balance}

Our goal is to group small pieces $W_1,\dots,W_\ell$
of the balanced bipartite graph $H$ on $2n$ vertices into packages
$P_1,\dots,P_k$ that form balanced bipartite subgraphs of~$H$. This is
equivalent to the following problem. Given the sizes $a_j$ and $b_j$ of the
colour classes of each piece $W_j$ (i.e., $a_j$ counts the
vertices of $W_j$ that are in $X$ and $b_j$ those that are in $Y$) we know
that the $a_j$'s sum up to $n$ and the $b_j$'s sum up to $n$. 
Then we would like to have a mapping $\phi:[\ell]\to[k]$ such that for all
$i\in[k]$ the $a_j$ with $j\in \phi^{-1}(i)$ sum up approximately to the same
value as the $b_j$ with $j\in \phi^{-1}(i)$. The following lemma asserts
that such a mapping $\phi$ exists. The package $P_i$ will then (in the proof of
Lemma~\ref{lem:H}) consist of all pieces $W_j$ with $j\in \phi^{-1}(i)$.

\begin{lemma}\label{lem:num} 
  For all $0 < \xi \leq 1/4$ and all positive integers $k$ 
  there exists $\ell\in\Nat$ such that for all integers $n\ge\ell$ the
  following holds. Let $(n_i)_{i\in [k]}$, $(a_j)_{j\in [\ell]}$, and
  $(b_j)_{j\in [\ell]}$ be integer partitions of $n$ such that $n_i\leq
  \frac{1}{8}n$ and $a_j+b_j \leq (1+\xi)\frac{2n}\ell$ for all $i\in [k]$,
  $j\in[\ell]$. Then there is a map $\phi:[\ell]\to[k]$ such that 
  for all $i\in [k]$ and $\sbar a_i:=\sum_{j\in \phi^{-1}(i)} a_j$ and 
  $\sbar b_i:=\sum_{j\in \phi^{-1}(i)} b_j$ we have
  \begin{equation}
    \label{lem:num:1}
    \sbar a_i < n_i + \xi n 
    \qquad\text{and}\qquad
    \sbar b_i < n_i + \xi n\,.      
  \end{equation}      
\end{lemma}

In the proof of Lemma~\ref{lem:num} we will use a Chernoff bound and
the following formulation of a concentration bound due to Hoeffding.

\begin{theorem}[Hoeffding bound~{\cite[Theorem~A.1.16]{TheProbMethod}}]
\label{thm:hoeff}
  Let $X_1,\dots,X_s$ be  independent random variables with $\Exp X_i=0$ and
  $|X_i|\le 1$ for all $i\in[s]$ and let $X$ be their sum. Then
  $\Prob[|X|\ge a]\le2\exp(-a^2/(2s))$.
\qed
\end{theorem}

\begin{proof}[Proof of Lemma~\ref{lem:num}]
  For the proof of this lemma we use a probabilistic argument and 
  show that under a suitable probability distribution a
  random map satisfies the desired properties with positive probability.
  For this purpose set $\ell := \bigl \lceil 1000 k^5 / \xi^2 \bigr
  \rceil $ and construct a random map $\phi\colon  [\ell] \rightarrow [k]$ by
  choosing $\phi(j)=i$ with probability $n_i /n$ for $i\in [k]$,
  independently for each $j\in [\ell]$. To show that this map
  satisfies~\eqref{lem:num:1} with positive probability
  we first estimate the sum of all $a_j$'s and $b_j$'s assigned to a fixed
  $i\in[k]$. To this end, let $\mathbbm 1_j$ be the indicator variable for the
  event $\phi(j)=i$ and define a random variable $S_i:=\sum_{j \in [\ell]}\mathbbm
  1_j$. Clearly $S_i$ is binomially distributed, we have $\Exp
  S_i=\ell\frac{n_i}n$, and by the Chernoff bound $\Prob[|S_i|\ge\Exp
  S_i+t]\le2\exp(-2t^2/\ell)$ (cf.~\cite[Remark~2.5]{purpleBook}) we get
  \begin{equation*}
    \Prob\Big[\big|S_i-\ell\frac{n_i}n\big|\ge\tfrac12\xi\ell\Big]
    \le2\exp(-\tfrac12\xi^2\ell).
  \end{equation*}
  Next, we examine the difference between the sum of the $a_j$'s assigned to
  $i$ and the sum of the $b_j$'s assigned to $i$. We define random variables
  $D_{i,j}:=\frac\ell{3n}(a_j -  b_j)(\mathbbm 1_j-\frac{n_i}n)$ and set $D_i
  :=\sum_{j\in[\ell]}D_{i,j}$. Then $\Exp D_{i,j}=0$ and as
  $a_j+b_j\le\frac{3n}{\ell}$ we have $|D_{i,j}|\le 1$. Thus
  Theorem~\ref{thm:hoeff} implies
  \begin{equation*}
    \Prob\big[|D_i|\ge\tfrac16\xi\ell\big] \le2\exp(-\tfrac1{72}\xi^2\ell).
  \end{equation*}
  By the union bound, the probability that we have 
  \begin{equation}\label{eq:num:good}
    |S_i-\ell\tfrac{n_i}n|<\tfrac12\xi\ell
    \qquad\text{and}\qquad  |D_i|<\tfrac16\xi\ell 
    \qquad\text{for all $i\in[k]$}
  \end{equation}
  is therefore at least $1-k\cdot2\exp(-\tfrac12\xi^2\ell)-
  k\cdot2\exp(-\tfrac1{72}\xi^2\ell)$ which is strictly greater than $0$ by our
  choice of $\ell$.
  Therefore there exists a map $\phi$ with~\eqref{eq:num:good}. We 
  claim that this map satisfies~\eqref{lem:num:1}. To see this, observe first
  that $\frac{3n}\ell D_i=\sum_{j\in\phi^{-1}(i)}(a_j-b_j)=\sbar a_i-\sbar b_i$
  which together with~\eqref{eq:num:good} implies $\sbar a_i - \sbar b_i < \xi n$.
  Moreover, we have $S_i=|\phi^{-1}(i)|$ and
  \begin{equation*}\begin{split} 
    \sbar a_i &= \tfrac{1}{2}(\sbar a_i+\sbar b_i) + \tfrac12(\sbar a_i-\sbar
    b_i) \le\tfrac{1}{2}(1+\xi)\tfrac{2n}\ell|\phi^{-1}(i)|
      +\tfrac12\cdot\tfrac12\xi n \\
    & \lByRef{eq:num:good}
      \frac{1}{2}(1+\xi)\frac{2n}\ell
        \Big(\ell\frac{n_i}n+\frac12\xi\ell\Big) 
       +\frac14\xi n 
    \le n_i+\xi n
  \end{split}\end{equation*}
  where the last inequality follows from  $\xi \leq \frac{1}{4}$ and $n_i\leq
  \frac{1}{8}n$. Since an entirely analogous calculation shows that $\sbar b_i 
  < n_i + \xi n$, this completes the proof of~\eqref{lem:num:1}.
\end{proof}

\subsection{The proof of the Lemma for~$H$}
\label{subsec:H}

For the proof of Lemma~$H$ we will now use Lemma~\ref{lem:num} as outlined in the
beginning of Section~\ref{subsec:balance}. In this way we obtain an assignment of
pieces $W_1,\dots\,W_\ell$ of~$H$ to edges $A_iB_i$ of $C$. This assignment,
however, does not readily give a homomorphism from~$H$ to~$C$ as there might be
edges between pieces $W_j$ and $W_{j+1}$ that end up on edges $A_iB_i$ and
$A_{i'}B_{i'}$ which are not neighbouring in~$C$. Nevertheless (owing to the
small bandwidth of~$H$) we will be able to transform it into a homomorphism by
assigning some few vertices of $W_{j+1}$ to other vertices of $C$ along the path
between $A_iB_i$ and $A_{i'}B_{i'}$ in~$C$.

\begin{proof}[Proof of Lemma~\ref{lem:H}]
Let $k$ and $\xi$ be given. Give $\xi' := \xi/4$ and $k$ 
to Lemma \ref{lem:num}, get $\ell$, set $\beta := \xi'/(4\ell k)$ 
and $n_0 := \lceil \ell /(2\xi) \rceil$, and let $H$ and $(n_i)_{i\in [k]}$ 
be given as in the statement of the lemma for~$H$. 

We assume that the vertices of $H$ are given a bandwidth labelling, 
partition $V(H)$ along this labelling into $\ell$ sets 
$W_1,\dotsc, W_\ell$ of as equal sizes as possible and 
define $x_i := |W_i\cap X|$ and $y_i := |W_i\cap Y|$. Then 
$x_i + y_i = |W_i|\leq \lceil 2n/\ell \rceil 
\leq 2n/\ell + 1 \leq (1+\xi) 2n/\ell$ and since $n_i\leq n/8$ by hypothesis 
we can give $(n_i)_{i\in [k]}$, $(x_i)_{i\in [\ell]}$ and 
$(y_i)_{i\in [\ell]}$ to Lemma \ref{lem:num} and 
get a $\phi\colon [\ell]\rightarrow [k]$ with~\eqref{lem:num:1}.

Let us discuss the main difficulty in our proof. Since the map $\phi$ 
is obtained via the probabilistic method, there is no control over how 
far apart in the Hamilton cycle $C$ two sets $W_{\varphi(i-1)}$ 
and $W_{\varphi(i)}$ will be assigned by $\varphi$. Hence
these sets might end up in non-adjacent vertices of the cycle~$C$.
If there are edges between $W_{\varphi(i-1)}$ and $W_{\varphi(i)}$ we
need to guarantee, however, that these edges are mapped to edges of~$C$ in
order to obtain the desired homomorphism~$f$.
Therefore, we resort to a greedy linking process which robs the pieces $W_i$ of
a small number of vertices. These are then distributed over the clusters lying
between the cluster pair $A_{\varphi(i-1)},B_{\varphi(i-1)}$ and the cluster
pair $A_{\varphi(i)},B_{\varphi(i)}$ such that the corresponding edges of~$H$
are placed on edges of~$C$.


Let $w_i$ be the first vertex in $W_i$ and define sets of \emph{linking vertices} by 
\[ L_j^i := [w_i+(j-1)\beta n, w_i+j\beta n)\subset W_i \] 
for every $j\in [2k]$ , and set $L^i := \bigcup_{j\in [2k]} L_j^i$. 
Then all $L_j^i$ have the common cardinality $\beta n$ and $|L^i|=2k\beta n$.
Since $\beta \leq 1/(4k\ell)$ implies that 
$2k\beta n + \beta n \leq \lfloor 2n/\ell \rfloor \leq |W_i|$ for every $i\in [\ell]$, 
we have $L^i\subsetneq W_i$ for every $i\in [\ell]$ where 
$|W_i\backslash L^i| \geq \beta n$, i.e., at the end of every 
set $W_i$ there are at least $\beta n$ non-linking vertices (see the left hand
side of Figure~\ref{fig:link}).


We now construct a map 
$f\colon V(H)\to\{A_1,\dotsc, A_k, B_1,\dotsc,B_k\}$ 
by defining, for every $i\in [\ell]$, 
\begin{equation}
f(x) :=
\begin{cases}
A_{\varphi(i-1) + \lfloor j/2 \rfloor } & \text{if $x\in L_j^i$ with $j\in \bigl [ 2\cdot \bigl( (\varphi(i)-\varphi(i-1)) \mod k \bigr)\bigr ]$,}\\
A_{\varphi(i)} & \text{else,}
\end{cases}
\end{equation}
for every $x\in W_i\cap X$, and
\begin{equation}
f(y) :=
\begin{cases}
B_{\varphi(i-1) + \lceil j/2 \rceil } & \text{if $y\in L_j^i$ with $j\in \bigl [ 2\cdot \bigl( (\varphi(i)-\varphi(i-1)) \mod k \bigr)\bigr ]$,}\\
B_{\varphi(i)} & \text{else,}
\end{cases}
\end{equation}
for every $y\in W_i\cap Y$, and show that this is indeed a homomorphism (see
also Figure~\ref{fig:link}). To do this, it is convenient to note that a set
$\{A_{i}, B_{i'} \}$ is an edge of $C$ if and only if $0 \leq i' - i \leq 1$.

\begin{figure}[ht]
\input{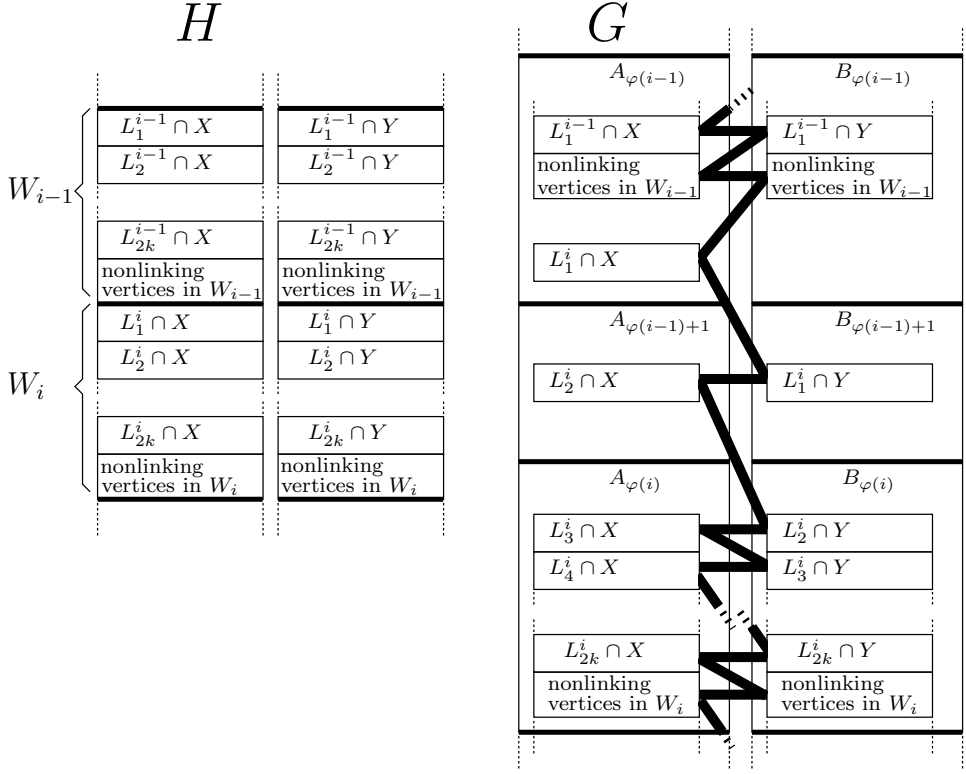}
\caption{The linking procedure.}
\label{fig:link}
\end{figure}

Let arbitrary vertices $x\in X$ and $y\in Y$ with $\{x,y\}\in F$ be given. 
Since the sets $W_i$ are defined along the bandwidth labelling, 
either $x$ and $y$ are both within the same $W_i$, or $x$ and $y$ lie 
in consecutive sets $W_i$ and $W_{i+1}$. We will now distinguish several 
cases. For brevity let 
$I_i :=  \bigl [ 2\cdot \bigl( (\varphi(i)-\varphi(i-1)) \mod k \bigr)\bigr ]$.

\smallskip

{\sl Case 1.}
Both $x$ and $y$ lie within the same set $W_i$.\\
{\sl Case 1.1.}
There is $j\in I_i$ with $x\in L_j^i$, hence $f(x) = A_{\varphi(i-1)+\lfloor j/2 \rfloor}$. 
Due to the bandwidth condition together with $|L_j^i|=\beta n$, if $y\notin L_j^i$ and $j+1\in I_i$, 
then necessarily $y\in L_{j+1}^i$, which explains the following three sub-cases.\\ 
{\sl Case 1.1.1.}
We have $y\in L_j^i$, hence $f(y) = B_{\varphi(i-1)+\lceil j/2 \rceil}$, hence 
the difference of the indices of $f(x)$ and $f(y)$ is $\lceil j/2 \rceil - \lfloor j/2 \rfloor$, which 
is either $0$ or $1$ according to whether $j$ is even or odd, hence $\{f(x),f(y)\}\in E(C)$.\\
{\sl Case 1.1.2.}
We have $y\notin L_j^i$ and $j+1\in I_i$, hence $y\in L_{j+1}^i$, hence 
$f(y) = \varphi(i-1)+\lceil (j+1)/2 \rceil$, hence the difference of indices 
of $f(y)$ and $f(x)$ is $\lceil (j+1)/2 \rceil - \lfloor j/2 \rfloor$, and this is 
always $1$, whether $j$ is even or odd, so $\{f(x),f(y)\}\in E(C)$.\\
{\sl Case 1.1.3.}
We have $y\notin L_j^i$ and $j+1\notin I_i$, 
hence $f(y) = B_{\varphi(i)}$. Here, 
$j+1\notin L_j^i$ implies that $j\geq 2\cdot \bigl( (\varphi(i)-\varphi(i-1)) \mod k \bigr)$ while 
being within Case 1.1 implies $j\in I_i$, hence $j\leq 2\cdot \bigl( (\varphi(i)-\varphi(i-1)) \mod k \bigr)$, 
so we have $j = 2\cdot \bigl( (\varphi(i)-\varphi(i-1)) \mod k \bigr)$, thus 
$f(x) = A_{\varphi(i-1)+\lfloor j/2 \rfloor} = A_{\varphi(i)}$, the index difference 
between $f(y)$ and $f(x)$ is $0$ and $\{f(x),f(y)\}\in E(C)$.\\
{\sl Case 1.2.}
There is no $j\in I_i$ with $x\in L_j^i$, hence $f(x) = A_{\varphi(i)}$. 
Being within Case 1, i.e. $y\in W_i$, it follows that there are exactly two cases.\\
{\sl Case 1.2.1.}
If $y$ precedes $x$ in the bandwidth labelling, then 
$y\in L_{2\cdot q}^i$ with $q=(\varphi(i)-\varphi(i-1)) \mod k$. Hence
$f(y)=B_{\varphi(i)}$, so the index difference between $f(y)$ and $f(x)$ is $0$
and $\{f(x),f(y)\}\in E(C)$.\\ 
{\sl Case 1.2.2.}
If $y$ succeeds $x$ in the bandwidth labelling, then, since $y\in W_i$ by 
being within Case 1, there is no $j\in I_i$ with $y\in I_i$, hence $f(y)=B_{\varphi(i)}$, 
so again the index difference between $f(y)$ and $f(x)$ is $0$ and $\{f(x),f(y)\}\in E(C)$.

\smallskip

{\sl Case 2.}
We have $x\in W_i$ and $y\in W_{i+1}$. Then, by the bandwidth condition 
and size of the sets of linking vertices, 
we must have $y\in L_1^{i+1}$, hence 
$f(y) = B_{\varphi((i+1)-1) + \lceil 1/2\rceil}  = B_{\varphi(i)+1}$, and since 
there are at least $\beta n$ non-linking vertices to the right of $W_i$, the 
vertex $x$ cannot lie in a $L_j^i$, hence $f(x) = A_{\varphi(i)}$, so the 
index difference of $f(y)$ and $f(x)$ is $1$ and $\{f(x),f(y)\}\in E(C)$.

\smallskip

{\sl Case 3.}
We have $y\in W_i$ and $x\in W_{i+1}$. Then, by the bandwidth condition 
and size of the sets of linking vertices, 
we must have $x\in L_1^{i+1}$, hence 
$f(x) = A_{\varphi((i+1)-1) + \lfloor 1/2\rfloor}  = A_{\varphi(i)}$, and since 
there are at least $\beta n$ non-linking vertices to the right of $W_i$, the 
vertex $y$ cannot lie in a $L_j^i$, hence $f(y)=B_{\varphi(i)}$, so the 
index difference of $f(y)$ and $f(x)$ is $0$ and $\{f(x),f(y)\}\in E(C)$. 
This completes the proof that $f$ is a homomorphism. 

\smallskip

We now prove \ref{lem:H:1} and \ref{lem:H:2}. Define $S := \bigcup_{i\in [\ell]} L^i$. Then  
$|S|\leq \ell\cdot 2k \cdot \beta n \leq \ell\cdot 2k \cdot (\xi'/(2\ell k))\cdot n = \xi' n\leq \xi n$, 
which shows \ref{lem:H:1}, and \ref{lem:H:2} is obvious from the 
definitions of $S$ and the map $f$ above. 

We now prove \ref{lem:H:3}. For this it suffices to note, rather crudely, that 
for every $j\in [k]$, no pre-image $f^{-1}(A_j)$ can become larger 
than the sum of the sizes of all sets $W_i$ assigned to $A_j$ by $\varphi$ 
(which by the definition of $f$ equals the sum of all $x_i = |X\cap W_i|$ 
with $\varphi(i)=j$)
plus the total number of linking vertices, i.e. for every $j\in [k]$, using 
the choice of $\beta$ and using that $\varphi$ has the property promised by Lemma \ref{lem:num}, we 
have $|f^{-1}(A_j)|\leq \bigl ( \sum_{i\in \varphi^{-1}(j)} x_i \bigr ) + |\bigcup_{i\in [\ell]} L^i| 
\leq n_j + \xi' n + \ell\cdot |L^i| = n_j + \xi' n + 2k\ell\beta n \leq n_j + 2\xi' n = n_j + \xi n$, 
completing the proof of \ref{lem:H:3}. 
\end{proof}

\section{Concluding remarks}

\noindent
{\bf Unbalanced~$H$ and~$G$.}
Essentially the same proof allows for an analogue of Theorem~\ref{thm:bipbw}
for bipartite graphs~$H$ and~$G$ that are not balanced but whose colour classes
have the same sizes. More precisely, let~$H=(X\dcup Y,F)$ and~$G=(A\dcup
B,E)$ be as in Theorem~\ref{thm:bipbw}, except that $|X|=|A|=n_1$ and
$|Y|=|B|=n_2$ (where $n_1+n_2=2n$) and the minimum degree condition on~$G$
is replaced by the following condition. For all $v\in A$ we have
$\deg_G(v)\ge(\frac12+\gamma)n_2$ and for all $w\in B$ we have
$\deg_G(w)\ge(\frac12+\gamma)n_1$. Then~$H$ is a subgraph of~$G$. 


\smallskip

\noindent
{\bf Generating systems for the cycle space.}
As an application of Theorem~\ref{thm:bipbw}, one can show the following result.
For every $\gamma > 0$ there is $n_0 \in\Nat$ such 
that for every $n\geq n_0$ every balanced bipartite graph $G$ on $2n$ vertices 
with $\delta(G)\ge(\frac12+\gamma)n$ has the property that the edge-sets of
all Hamilton cycles in $G$ form a generating system for the cycle space of $G$.
A proof for this theorem will be given in a forthcoming
paper~\cite{heinighamspace}. It utilises the fact that a special balanced
bipartite graph~$H$ (the so-called M\"obius ladder) of bounded maximum degree
and bandwidth has this property and then shows that this gets translated to the
graph~$G$, using a result of Locke~\cite{locke85}.

\bibliographystyle{amsplain} \bibliography{bipartite}


\end{document}